\documentclass[12pt]{amsart}
\usepackage{amsmath,amssymb}
\usepackage{eufrak}
\usepackage{amsthm}
\usepackage{enumitem}
\usepackage{graphicx}
\usepackage{calc}
\usepackage{graphics}
\usepackage{color}
\usepackage[dvipsnames]{xcolor}

\usepackage{youngtab}
\usepackage{marginnote}
\usepackage{lscape}

\usepackage[perpage]{footmisc}

\usepackage{blindtext}

\usepackage{tikz}
\usetikzlibrary{calc,shadings,patterns}
\usetikzlibrary{matrix}
\usepackage{faktor}
\usepackage[normalem]{ulem}
\usepackage{bm}
\usepackage{afterpage}
\usepackage{mathrsfs}
\usepackage[all]{xy}
\usetikzlibrary{arrows}
\usetikzlibrary{matrix}
\usepackage{verbatim}
\usepackage{setspace}
\usepackage{multirow}
\usetikzlibrary{decorations.pathreplacing}
\usepackage{enumitem,pifont,xcolor}
\definecolor{myblue}{RGB}{0,29,119}

\newtheorem{theorem}{Theorem}[section]
\newtheorem{proposition}[theorem]{Proposition}
\newtheorem{corollary}[theorem]{Corollary}
\newtheorem{lemma}[theorem]{Lemma}
\theoremstyle{definition}
\newtheorem{definition}[theorem]{Definition}
\newtheorem{example}[theorem]{Example}

\newtheorem{remark}[theorem]{Remark}

\newtheorem*{theorem*}{Theorem}

\makeatletter
\@addtoreset{equation}{section}
\makeatother

\newcommand{\Cc}{{\mathbb C}} 
\DeclareMathOperator{\sign}{sign}
\newcommand{\hH}{\mathfrak{h}}
\newcommand{\gG}{\mathfrak{g}}

\newcommand{\cK}{{\mathcal K}}
\newcommand{\cL}{{\mathcal L}}

\DeclareMathOperator{\coker}{coker}
\DeclareMathOperator{\im}{Im}

\providecommand{\AMS}{$\mathcal{A}$\kern-.1667em%
\lower.25em\hbox{$\mathcal{M}$}\kern-.125em$\mathcal{S}$}

\usepackage{footnote}

\newcommand{\astfootnote}[1]{
\let\oldthefootnote=\thefootnote
\setcounter{footnote}{0}
\renewcommand{\thefootnote}{\fnsymbol{footnote}}
\footnote{#1}
\let\thefootnote=\oldthefootnote
}
\usepackage{hyperref}
\hypersetup{ colorlinks,
 linkcolor={blue},  citecolor={red}, urlcolor={blue}}

\begin{document}

{\let\thefootnote\relax\footnotetext{\!\!\!\!\!\!\!\!\!\!\!\!\!\!\! MSC 2010:\,\,\,\,17B30, 17B55, 17B56, 05E15
Contact: \href{mailto:emre-sen@uiowa.edu}{emre-sen@uiowa.edu} }}


\title{on the Homology of Nilpotent k-ary  Lie Algebras}
\begin{abstract} We introduce nilpotent k-ary Lie algebras including analogues of Heisenberg Lie algebras and free nilpotent Lie algebras.
We study homology of k-ary nilpotent Lie algebras by using a modification of Chevalley-Eilenberg complex. For some classes of nilpotent k-ary Lie algebras and in particular Heisenberg k-ary Lie algebras we give explicit formulas for Betti numbers. Representation stability of free nilpotent k-ary Lie algebras is proven and lower bounds for Betti numbers are described by Schur modules. We also verify that toral rank conjecture holds for the classes we studied. Moreover, for 2-step nilpotent k-ary Lie algebras, we give a refinement of it.

\end{abstract}

\author{Emre SEN}



\maketitle


\section{Introduction}

k-ary Lie algebra $\mathcal{L}$ is a vector space equipped with k-bracket
\begin{center}
$[.,\ldots,.]:\cL\times\ldots\times\cL\longrightarrow \cL$
\end{center} 
satisfying:
\begin{enumerate}[label=\roman*)]
\item antisymmetry i.e. $[x_1,\ldots,x_k]=sgn(\sigma)[x_{\sigma(1)},\ldots,x_{\sigma(k)}]$ where $\sigma\in\mathfrak{S}_k$ and
\item \label{jacobi} generalized Jacobi identity:
\begin{center}
$[[x_1,x_2\ldots,x_k],x_{k+1},\ldots,x_{2k-1}]=\sum\limits^{k}_{i=1}[x_1,x_2,\ldots,x_{i-1},[x_i,x_{k+1},\ldots,x_{2k-1}],x_{i+1},\ldots,x_k]$
\end{center} 
\end{enumerate}

It appears in the works of Filippov \cite{filippov26n} and Hanlon, Wachs \cite{hanlon1995lie}.
We can define \emph{lower central series} for k-ary Lie algebras as:
\begin{center}
$\cL\supset \underbrace{\left[\cL,\ldots,\cL\right]}_{k}\supset \underbrace{\left[\underbrace{\left[\cL,\ldots,\cL\right]}_{k},\cL,\ldots,\cL\right]}_{2k-1}\supset\ldots$
\end{center}
We call k-ary Lie algebra \emph{nilpotent} if lower central series terminates. In particular if the bracket of size greater than $k$ vanishes we call them \emph{ $2$-step nilpotent k-ary Lie algebras}. In other words, $\cL$ is 2-step nilpotent k-ary Lie algebra if and only if k-commutator $[\cL,\ldots,\cL]$ is contained in the center of $\cL$.
The aim of this paper is to introduce some classes of k-ary nilpotent Lie algebras and describe their homology by using a generalization of Chevalley-Eilenberg complex. Even in the classical set up $k=2$, computation of Betti numbers is a hard task, there are only limited cases that we know them explicitly. Most of them are dedicated for 2-step nilpotent (super) Lie algebras \cite{armstrong1997explicit}, \cite{sigg1996laplacian}, (\cite{getzler1999homology}), Heisenberg Lie algebras \cite{santrab} and some free nilpotent Lie algebras of small dimensions \cite{tirao2002homology}. Also, recently, it can be viewed as a problem of representation stability in the sense of Church and Farber \cite{church2013representation}. From this perspective homology of 2-step nilpotent Lie algebras, super Lie algebras and their analogues studied by S. Steven \cite{sam2013homology}.

Thus, we focus on k-ary generalizations of the above cases i.e. k-ary Heisenberg Lie algebras, free nilpotent k-ary Lie algebras etc. aiming to carry results on the homology of nilpotent Lie algebras into homology of nilpotent k-ary Lie algebras.

A real Heisenberg Lie algebra $H$ is $2m+1$ dimensional algebra with basis \\$\{x_1,\ldots,x_m,y_1,\ldots,y_m,z\}$ satisfying $[x_i,y_j]=\delta_{i,j}z$. Its homology is given in \cite{santrab}: 
\begin{align*}
\vert H^i{(H)}\vert={2m\choose i}-{2m\choose i-2}
\end{align*}

We consider the following generalization of it: let $\hH$ be $km+1$ dimensional algebra with basis $\{x^1_{1},\ldots,x^1_{m},x^2_1,\ldots,x^2_m,\ldots,x^k_{m},z\}$ such that:
\begin{center}
$[x^1_{i},x^2_{i},\ldots, x^k_{i}]=z$
\end{center}
We prove that:
\begin{theorem}
Betti numbers are:
\begin{align}
\vert H^{ik-(i-1)}\left(\hH\right)\vert={km\choose i(k-1)+1}-{km\choose (i-1)(k-1)}
\end{align}
\end{theorem}

Another class of 2-step nilpotent Lie algebras was introduced by Armstrong, Cairns and Jessup in \cite{armstrong1997explicit}. We define k-ary analogue of it (see definition \ref{defACJ}), and study their homology in section \ref{homACJ}.\\

The following is a generalization of free Lie algebra for k-ary Lie algebras:
\begin{definition}\cite{friedmann2017generalization} The free k-ary Lie algebra on the set $X$ is a k-ary Lie algebra $\cL$ together with mapping $e:X\mapsto \cL$ with the following universal property: for each k-ary Lie algebra $\cK$ and for each mapping $f:X\mapsto \cK$ there is a unique k-ary Lie algebra homomorphism $F:\cL\mapsto\cK$ such that $f=F\circ e$
\end{definition}

The reason why we consider free k-ary Lie algebras is due to the well-known construction. Let $V$ be a vector space, $T(V)$ be its  tensor algebra. Let $F(V)=\{v\otimes u-u\otimes v\vert u,v\in T(V)\}$. $F(V)$ is free Lie algebra and it has a natural grading with respect to degree. Graded piece $F(V)_d$ of degree $d$ becomes both representation of symmetric group $\mathfrak{S}_d$ and $GL(V)$ which creates a rich combinatorics.

 The question of which irreducible representations appear in $F(V)_d$ solved by Klyachko \cite{klyachko1974lie}. 
 The first few summands of $F(V)$ are:
\begin{gather}\label{freelie}
F(V)=V\oplus S_{11}V \oplus S_{21}V\oplus \left(S_{31}V\oplus S_{21^2}V\right)\oplus \cdots 
\end{gather}
where $S_{\lambda}V$ is representation parametrized by Young diagram $\lambda$. 
Free nilpotent Lie algebras are quotients of $F(V)$ i.e. taking quotient by elements with degrees greater then $d$ gives d-step nilpotent free Lie algebra. For example, 2-step free nilpotent Lie algebra is
\begin{align}
\mathfrak{n}_2=V\oplus \bigwedge^2 V,\hspace{1cm} [x,y]=x\wedge y
\end{align}
As an algebra, their structure is relatively simple but homologies are complicated, only 2-step case $\mathfrak{n}_2$ was solved, see \cite{jozefiak1988representation}, \cite{sigg1996laplacian}, \cite{sam2013homology}. 


The decomposition of free k-ary Lie algebras into irreducible components is still an open problem, there are particular results    \cite{friedmann2017generalization}. For example, the second and the third graded components are parametrized by diagrams $\lambda=(1^k)$ and $\lambda=(2^{k-1},1)$ respectively.
We study homologies of free 2-step nilpotent k-ary Lie algebras:
\begin{align}
\mathfrak{n}^k_2=V\oplus\bigwedge^kV,\hspace{1cm} [x_1,x_2,\ldots,x_k]=x_1\wedge x_2\wedge\cdots\wedge x_k
\end{align}

We prove that each homology is a Schur module. However, we do not have general description of homologies, because to find irreducible summands of the complex $\left(\bigwedge\mathfrak{n}^k_2\right)$, the plethsym problem i.e. decomposition of $\bigwedge^j(\bigwedge^kV)$, $j,k\geq 3$ has to be solved. Nevertheless, we study lower bounds of Betti numbers in section \ref{sectionFree}.

For small dimensional case $\dim V=k$ we compute the homology of 3-step nilpotent k-ary Lie algebras. In section \ref{homfree3}, we show:
\begin{theorem} The ranks of homologies of free 3-step nilpotent k-ary Lie algebra $\mathfrak{n}^k_3$ of dimension $k+1$ are:
\begin{gather*}
\vert H^1(\mathfrak{n}^k_3)\vert=k\\
\vert H^k(\mathfrak{n}^k_3)\vert={2k+1\choose k}-\left(3k+2\right)\\
\vert H^{2k-1}(\mathfrak{n}^k_3)\vert=\left(2k+1\right)\left(k-1\right)
\end{gather*}

\end{theorem}

Beyond the explicit computation of Betti numbers, there are questions concerning the total homology of Lie algebras. For instance, the toral rank conjecture states that for any nilpotent Lie algebra $\mathfrak{n}$, the inequality
 \begin{align}
\sum\limits^{\dim\mathfrak{n}}_{i=0} H^i\left(\mathfrak{n}\right)\geq 2^{\dim\mathfrak{z}}
\end{align} 
is valid where $\mathfrak{z}$ is the center of $\mathfrak{n}$. We verify that this holds for the algebras (definitions  \ref{defheis},\ref{defACJ},\ref{deffree2}, \ref{deffree3}) which are subject to this work In particular, for 2-step nilpotent k-ary Lie algebras, we improve the lower bound of total homology in theorem \ref{thmRefinement} by generalizing arguments of \cite{tirao2000refinement}.

Another open question about Lie algebra homology is the property M of Hanlon \cite{hanlon}, which relates total homology of current Lie algebra to homology of Lie algebra itself :

\begin{align}\label{defPropertyM}
H^*\left(\gG\otimes \faktor{\Cc[t]}{t^j}\right)=H^*\left(\gG\right)^{\otimes j}
\end{align}

For example, semisimple Lie algebras satisfy it. However, there are counterexamples to this conjecture, since then the purpose is to find algebras satisfying \ref{defPropertyM} above. Even the case of Heisenberg Lie algebras (k=2) of dimension greater or equal than $5$ remains unknown \cite{hanlon2002property}. We show that for k-ary Heisenberg Lie algebras with $k\geq 5$ Hanlon's conjecture \ref{defPropertyM} is not true.

Organization of the paper is : first we give definitions of some classes of k-ary Lie algebras and compare them with the classical (k=2) case in Section \ref{sectionExamples}. In the \ref{defResolution}rd section  we study their homologies. For k-ary Heisenberg Lie algebras, we give explicit Betti numbers. In section \ref{sectionFree}, we obtain lower bounds of Betti numbers and investigate which Schur modules can appear for some particular cases of free 2-step nilpotent k-ary Lie algebras. We discuss representation stability of some cases in \ref{sectionRep}. The last section \ref{sectionToral} is about toral rank conjecture  of nilpotent k-ary Lie algebras.

\section{Examples of Nilpotent k-ary Lie Algebras}\label{sectionExamples}
Here we present some examples of nilpotent k-ary Lie algebras:

\subsection{Generalization of Heisenberg Lie algebra}
The algebra H with basis $\{x,y,z\}$ and relation $[x,y]=z$ is called Heisenberg Lie algebra which is a central subject in many fields. We consider the following generalization of it:
\begin{definition}\label{defheis} A $km+1$ dimensional algebra $\hH$ with basis 
\begin{center}
$\{x^1_{1},\ldots,x^1_{m},x^2_1,\ldots,x^2_m,\ldots,x^k_{m},z\}$
\end{center}
 and with the only nonzero k-bracket (up to permutation):
\begin{center}
$[x^1_{i},x^2_{i},\ldots x^k_{i}]=z$
\end{center}
is called \emph{k-ary Heisenberg Lie algebra}.
\end{definition}
It is clear that $[\hH,\ldots,\hH]=z$ which implies that $\hH$ is 2-step nilpotent k-ary Lie algebra.
\subsection{A class of nilpotent-ACJ}\label{defACJ} The following is generalization of the algebra studied in \cite{armstrong1997explicit}. Let $\gG$ be Lie algebra with basis $\{z,x^1_1,\ldots,x^1_m,x^2_1,\ldots x^2_m,\ldots,x^{k}_1,\ldots, x^k_m\}$ and with the bracket:
\begin{align}
[z,x^1_i,x^2_i,\ldots,x^{k-1}_i]=x^k_i
\end{align}
Center of the algebra is $\left\{x^k_1,\ldots,x^k_{m}\right\}$ which is equal to k-commutator of the algebra.

\subsection{Free 2-step nilpotent k-ary Lie algebras}\label{deffree2}
Let $V$ be vector space over $\Cc$ with $\dim V\geq k$. We consider $\mathfrak{g}=V\oplus\bigwedge^kV$ with respect to bracket:

\begin{center}
$\underbrace{\gG\times\cdots\times\gG}_{k-many}\mapsto \gG$\\
$[v_1,\ldots,v_k]=\begin{cases} 
v_1\wedge\ldots\wedge v_k &\text{each}\quad v_i\in V \\
\quad\quad 0 &\text{otherwise}
\end{cases}$
\end{center}

Since the bracket is antisymmetric and satisfies generalized Jacobi identity identically, $\gG$ is free k-ary Lie algebra. Moreover, it is nilpotent, since $[\gG,\ldots,\gG]=\bigwedge^kV$. 

\subsection{Free 3-step nilpotent Lie algebras}\label{deffree3}
To illustrate the main idea here, we choose $k=3$. Let $X$ be set, $X=\left\{x,y,z\right\}$. We use the Hall Basis $\left\{x,y,z,[xyz],[xy[xyz]],[yz[xyz]],[xz[xyz]]\right\}$ for the free 3-ary Lie algebra $\cL$ on the set $X$, where $[,,]$ is 3-bracket. It is nilpotent since the term $\left[\left[\left[\left[\cL,\cL,\cL\right],\cL,\cL\right],\cL,\cL\right],\cL,\cL\right]$ of lower central series vanishes.

\section{Chevalley-Eilenberg Complex and Homology}\label{defResolution}
Let $\gG$ be k-ary Lie algebra over $\Cc$. We use the below chain complex which mimics Chevalley-Eilenberg complex:

\begin{gather}
\begin{aligned}
\xymatrixcolsep{5pt}
\xymatrix{& \ar[rr]&&\bigwedge^{3k-2}\gG\ar[rrrr]^{\partial_{3k-2}}  &&&& \bigwedge^{2k-1}\gG\ar[rrrr]^{\partial_{2k-1}} &&&& \bigwedge^{k}\gG\ar[rrrr]^{\partial_{k}}  &&&&  \gG\ar[rr] &&\Cc
}
\end{aligned}
\label{defcomplex}
\end{gather}
with the differential 

\begin{align}\label{defdiff}
\partial_t(x_1\wedge\ldots\wedge x_t)=\sum\limits_{\sigma\in D_{t,k}}sgn(\sigma)[x_{\sigma(1)},\ldots,x_{\sigma(k)}]\wedge x_{\sigma(k+1)}\wedge \cdots \wedge x_{\sigma(t)}
\end{align}

where
\begin{align}
D_{t,k}=\left\{\sigma\in\mathfrak{S}_t\,\,\vert\,\, \sigma(1)<\sigma(2)<\cdots<\sigma(k)\quad\text{and}\quad\sigma(k+1)<\cdots<\sigma(t)\right\}
\end{align}

This complex appears in \cite{hanlon1995lie}, and they compute homology of free k-ary Lie algebra $F(V)$. We just refer to their result:
\begin{proposition} $\partial^2=0$
\end{proposition}
\begin{proof}
See \cite{hanlon1995lie} theorem 1.6.
\end{proof}

Therefore, we can use chain complex \ref{defcomplex}, as generalization of Chevalley-Eilenberg complex, hence we define homology with trivial coefficients according to \ref{defdiff} i.e. 
\begin{align}\label{defhomology}
H^{ik-i+1}(\gG)={\Large \cfrac{\ker\partial_{ik-i+1}}{\im\partial_{(i+1)k-i}}}
\end{align}

\subsection{Homology of k-ary Heisenberg Lie Algebras}\label{homHeis}

We prove the following:

\begin{theorem}\label{thmHeis} Betti numbers of k-ary Heisenberg Lie algebra $\hH$ are:

\begin{align}
\left|H^{ik-(i-1)}\left(\hH\right)\right|={km\choose i(k-1)+1}-{km\choose (i-1)(k-1)}
\end{align}
\end{theorem}
\begin{corollary}  In particular, when $k=2$ this recovers the well known result of \cite{santrab}.
\end{corollary}

Since $\left|H^{ik-(i-1)}\left(\hH\right)\right|=\left|\ker\partial_{ik-(i-1)}\right|-\left|\im\partial_{(i+1)k-i}\right|$ and $\dim\bigwedge^{ik-(i-1)}\left(\hH\right)=\left|\ker\partial_{ik-(i-1)}\right|+\left|\im\partial_{ik-(i-1)}\right|$ it suffices to compute the dimension of the images of the boundary maps. Answer is:
\begin{proposition}\label{thmHeisImage} The dimension of the image of $\partial_{ik-(i-1)}$ in $\bigwedge^{(i-1)k-(i-2)}\hH$ is 
\begin{align}\label{heisenbergformula}
{km\choose (i-1)(k-1)}
\end{align}
\end{proposition}

\begin{proof}
We prove it by induction on the dimension of $\hH$. For simplicity, first we analyze the case $k=3$. Then we prove it for arbitrary $k$.\\
Let $k=3$, $m=1$. Hence $\hH$ has basis $\{x,y,z,w\}$ with bracket $[x,y,z]=w$. Therefore $\vert\im \partial_{3}\vert =1$ which is the value of \ref{heisenbergformula} at $i=1$.
Assume that the formula \ref{heisenbergformula} holds for $m$. Let $\hH'=\hH\oplus \Cc^3$ where $\Cc^3=<x,y,z>$. Now 

\begin{align}
\bigwedge^{2i+1}\hH'=\bigwedge^{2i+1}\hH\oplus\left(\bigwedge^{2i}\hH\wedge \Cc^3\right)\oplus\left(\bigwedge^{2i-1}\hH\wedge\bigwedge^2\Cc^3\right)\oplus\left(\bigwedge^{2i-2}\hH\wedge\bigwedge^3\Cc^3\right)
\end{align}
The map $\partial^{'}_{2i+1}:\bigwedge^{2i+1}\hH'\mapsto \bigwedge^{2i-1}\hH'$ decomposes as:
\begin{align}
\partial^{'}_{2i+1}=\partial_{2i+1}\oplus\left(\partial_{2i}\otimes\Cc^3\right)\oplus\left(\partial_{2i-1}\otimes\bigwedge^2\Cc^3\right)\oplus\left(\partial_{2i-2}\otimes\bigwedge^3\Cc^3\right)
\end{align}
Therefore 
\begin{gather}
\dim\!\im\partial^{'}_{2i+1}\!=\!\dim\!\im\partial_{2i+1}\!+\!\dim\!\im\left(\partial_{2i}\otimes\Cc^3\right)+\dim\!\im\left(\partial_{2i-1}\otimes\bigwedge^2\Cc^3\right)+\dim\!\im\left(\partial_{2i-2}\otimes\bigwedge^3\Cc^3\right)=\nonumber\\\dim\im\partial^{'}_{2i+1}=\dim\im\partial_{2i+1}+3\dim\im\partial_{2i}+3\dim\im\partial_{2i-1}+\dim\im\partial_{2i-2}\\
\dim\im\partial^{'}_{2i+1}={3m\choose 2i-2}+{3m\choose 2i-3}.3+{3m\choose 2i-4}.3+{3m\choose 2i-5}\\
\dim\im\partial^{'}_{2i+1}={3m+3\choose 2i-2}
\end{gather}
This finishes the proof for $k=3$

We follow similar arguments to prove the formula \ref{heisenbergformula} for arbitrary $k\geq 4$. 
Let $m=1$, hence $\hH$ has basis $\{x^1,\ldots,x^k,w\}$ with bracket $[x^1,\ldots,x^k]=w$. The size of the image of $\partial_k$ is $1$ which coincides with \ref{heisenbergformula} at $i=1$. Assume that formula holds for $m$. Let $\hH'=\hH\oplus\Cc^k$ where $\Cc^k=\left\langle x^1,\ldots x^k \right\rangle$. We have the following decomposition of $\bigwedge^{ik-(i-1)}\hH'$ and the map $\partial_{ik-(i-1)}$:

\begin{align}
\bigwedge^{ik-(i-1)}\hH'=\bigoplus\limits^{k}_{j=0}\left(\bigwedge^{ik-(i-1)-j}\hH\wedge\bigwedge^j\Cc^k\right)
\end{align}

\begin{align}
\partial^{'}_{ik-(i-1)}=\bigoplus\limits^{k}_{j=0}\left(\partial_{ik-(i-1)-j}\otimes \bigwedge^{j}\Cc^{k}\right)
\end{align}

Therefore:
\begin{gather}
\dim\im\partial^{'}_{ik-(i-1)}=\sum\limits^{k}_{j=0}\dim\im\left(\partial_{ik-(i-1)-j}\otimes \bigwedge^{j}\Cc^{k}\right)=\\
\dim\im\partial^{'}_{ik-(i-1)}=\sum\limits^{k}_{j=0}\dim\im\left(\partial_{ik-(i-1)-j}\right){k\choose j}=\\
\dim\im\partial^{'}_{ik-(i-1)}=\sum\limits^{k}_{j=0}{km\choose ik-(i-1)-j}{k\choose j}=\\
\dim\im\partial^{'}_{ik-(i-1)}={km+k\choose ik-(i-1)}
\end{gather}
This finishes the proof.
\end{proof}

By using the proposition \ref{thmHeisImage}, we present the proof of the theorem \ref{thmHeis}:

\begin{proof}
Since $\dim\bigwedge^{ik-(i-1)}\left(\hH\right)=\left|\ker\partial_{ik-(i-1)}\right|+\left|\im\partial_{ik-(i-1)}\right|$, the size of the kernel is:
\begin{align}
{km+1\choose ik-(i-1)}-{km\choose (i-1)(k-1)}
\end{align}
Therefore by using $\left|H^{ik-(i-1)}\left(\hH\right)\right|=\left|\ker\partial_{ik-(i-1)}\right|-\left|\im\partial_{(i+1)k-i}\right|$ we get:

\begin{gather}
\left|H^{ik-(i-1)}\left(\hH\right)\right|={km+1\choose ik-(i-1)}-{km\choose (i-1)(k-1)}-{km\choose i(k-1)}=\\
{km\choose ik-(i-1)}+{km\choose ik-i}-{km\choose (i-1)(k-1)}-{km\choose i(k-1)}=\\
{km\choose ik-(i-1)}-{km\choose (i-1)(k-1)}
\end{gather}
or equivalently 
\begin{align}
{km\choose i(k-1)+1}-{km\choose (i-1)(k-1)}
\end{align}
since $ik-(i-1)=i(k-1)+1$, which completes proof.
\end{proof}

\begin{remark} The formula is valid with constraints on the dimensions, there is no Poincare type duality theorem for the complex \ref{defcomplex}. Simply, we can choose 
\begin{gather}
\alpha\leq \lfloor\cfrac{km+1}{2}\rfloor
\end{gather}
since we want $\bigwedge^{\alpha}\hH$ be the largest dimensional space, where $\alpha=ik-(i-1)$ and $\lfloor-\rfloor$ is floor function.
Substitution gives:
\begin{gather}
i\leq \cfrac{1}{k-1}\left(\lfloor\cfrac{km+1}{2}\rfloor-1\right)
\end{gather}
\end{remark}

\subsection{Homology of a k-ary ACJ Lie algebra}\label{homACJ}
Let $\gG$ be Lie algebra with basis\\ $\{z,x^1_1,\ldots,x^1_m,x^2_1,\ldots x^2_m,\ldots,x^{k}_1,\ldots, x^k_m\}$ and with the bracket:
\begin{align}
[z,x^1_i,x^2_i,\ldots,x^{k-1}_i]=x^k_i
\end{align}
The case  when $k=2$ were introduced by Armstrong, Cairns and Jessup in \cite{armstrong1997explicit}. $i$th Betti number is
\begin{align}
{m+1\choose \lfloor\frac{i+1}{2}\rfloor}{m\choose\lfloor\frac{i}{2}\rfloor}
\end{align}
where $\lfloor i\rfloor$ denotes the integer part of $i$. To compute it, they use derivation arising from adjoint map and then solve a recursion. We consider a similar reduction.

A k-ary analogue of adjoint map can be made in many ways, we choose:
\begin{gather}
ad(z):\bigwedge^{k-1}\gG\mapsto\gG\\
ad(z)(x_1\wedge\cdots\wedge x_{k-1})=[z,x_1,\ldots,x_k]
\end{gather}

The ACJ type k-ary Lie algebras given in \ref{defACJ} can be seen as a specific example of more general set up. Let $\gG$ be k-ary Lie algebra such that $\mathfrak{a}$ is codimension one abelian ideal in $\gG$ i.e. 
\begin{gather}
\gG=\langle z\rangle\oplus\mathfrak{a}\\
[\gG,\ldots,\gG,\mathfrak{a}]\subseteq\mathfrak{a}
\end{gather}

Then:
\begin{proposition}
Homology of $\gG$ at $\alpha=ik-(i-1)$ is:
\begin{gather}\label{f1}
\vert H^{\alpha}(\gG)\vert=
{ km\choose \alpha}-{km\choose \alpha+k-2}+\dim\ker\theta_{\alpha-1}+\dim\ker\theta_{\alpha+k-2}
\end{gather}
where $\theta_{\alpha}:=ad(z)\left(\bigwedge^{\alpha}\mathfrak{a}\right)$.
\end{proposition}

\begin{proof}
We can express boundary maps in terms of $\theta$:
\begin{gather}
\vert H^{\alpha}(\gG)\vert=\vert\ker\partial_{\alpha}\vert-\vert\im\partial_{\alpha_{\alpha+k-1}}\vert=\\
\dim\bigwedge^{\alpha}\mathfrak{a}+\dim\ker\theta\left(\bigwedge^{\alpha-1}\mathfrak{a}\right)-\dim\theta\left(\bigwedge^{\alpha+k-2}\mathfrak{a}\right)=\\
\dim\bigwedge^{\alpha}\mathfrak{a}+\dim\ker\theta\left(\bigwedge^{\alpha-1}\mathfrak{a}\right)+\dim\ker\theta\left(\bigwedge^{\alpha+k-2}\mathfrak{a}\right)-\dim\bigwedge^{\alpha+k-2}\mathfrak{a}=\\
{\dim \mathfrak{a}\choose \alpha}-{\dim \mathfrak{a}\choose \alpha+k-2}+\dim\ker\theta_{\alpha-1}+\dim\ker\theta_{\alpha+k-2}
\end{gather}
The expression \ref{f1} holds.
\end{proof}

Therefore computation of the homology reduces to the computation of the dimension of the kernels of $\theta$.
The action of $\theta$ is :
\begin{gather}
\theta_{j}:\bigwedge^j\mathfrak{a}\mapsto\bigwedge^{j-k+2}\mathfrak{a}\\
\theta_j\left(x_1\wedge\cdots\wedge x_j\right)=\sum\limits_{i_1,\ldots, i_{j}} ad(z)(x_{i_1}\wedge \cdots\wedge x_{i_{j}})\wedge x_1\wedge \cdots\wedge\widehat{x_{i_a}}\wedge\cdots\wedge x_{j}=\\
\sum\limits_{i_1,\ldots, i_{j}}[z,x_{i_1},\ldots x_{i_{j}}]\wedge x_1\wedge \cdots\wedge\widehat{x_{i_a}}\wedge\cdots \wedge x_{j}
\end{gather}

We do not know how to find $\ker \theta_j$, since they are subject to recursive relations, and even the case $k=2$ requires work \cite{armstrong1997explicit},\cite{armstrong1996cohomology}. By straightforward computation, we only give the second homology:

\begin{proposition}
The second homology of ACJ type Lie algebra $\gG$ is 
\begin{gather}
\left| H^k(\gG)\right|={km+1\choose k}-m{km-k\choose k-1}-{m+1\choose 2}
\end{gather}
\end{proposition}
\begin{proof}
Consider the sequence:
\begin{align}
\bigwedge^{2k-1}\gG\rightarrow\bigwedge^k\gG\rightarrow \gG
\end{align}
Then the image of $\partial_k$ is $\left\{x^k_1,\ldots,x^k_m\right\}$. Therefore the size of the kernel of $\partial_k$ is
\begin{align}
{km+1\choose k}-m
\end{align}
The possible contributors to image are the cycles of the forms:
\begin{enumerate}[label=\roman*)]
\item Cycles containing 2-pairs i.e. $z\wedge x^i_1\wedge\cdots\wedge x^i_{k-1}\wedge x^j_{1}\wedge\cdots\wedge x^j_{k-1}$, in total this gives ${m\choose 2}$ elements in $\bigwedge^k\gG$.
\item Cycles containing 1 pair i.e. $z\wedge x^i_1\wedge\cdots\wedge x^i_{k-1}\wedge *\wedge\cdots\wedge *$. The image is of the form $x^i_k\wedge *$, the total number is 
\begin{align}
m\left({km-k+1\choose k-1}-{km-k\choose k-2}\right)
\end{align}
Therefore homology is 
\begin{gather}
{km+1\choose k}-m-m\left({km-k+1\choose k-1}-{km-k\choose k-2}\right)-{m\choose 2}=\\
{km+1\choose k}-m{km-k\choose k-1}-{m+1\choose 2}
\end{gather}
\end{enumerate}
\end{proof}

\subsection{Homology of 3-step nilpotent k-ary Lie algebras of small rank}\label{homfree3}

Let $\gG$ be $3$-step nilpotent free k-ary Lie algebra on $V$ with $\dim V=k$ where $k\geq 4$. We fix Hall basis of $\gG$ as:\\
There are $k$ elements $x_1,\ldots,x_k$ of degree $1$\\
There is one element $y=[x_1\ldots x_k]$ of degree $k$\\
There are $k$ elements $[x_1,\ldots,\widehat{x_i},\ldots,y] $ of degree $2k-1$.

The complex below computes homology:
\begin{align}\label{f2}
\xymatrixcolsep{5pt}
\xymatrix{0\ar[rr]&&\bigwedge^{2k-1}\gG\ar[rrrr]^{\partial_{2k-1}} &&&& \bigwedge^{k}\gG\ar[rrrr]^{\partial_{k}}  &&&&  \gG \ar[rr] && \Cc
}
\end{align}

Image of $\partial_{k}$ is explicitly degree $k$ and $2k-1$ terms, hence $\vert\im\partial_k\vert=1+k$ and $\vert\ker\partial_k\vert={2k+1\choose k}-(k+1)$\\
\begin{lemma} The image of $\partial_{2k-1}$ is of size $k+1$.
\end{lemma}
\begin{proof}
Elements of $\bigwedge^{2k-1}\gG$ which does not contain at least $k-1$ degree one terms are in the kernel. In total we have four possibilities:
\begin{enumerate}[label=\roman*)]
\item All linear terms are in cocycle, no $y$, gives cocyles of $y$ and degree $2k-1$ elements, total size is ${k\choose k-1}$
\item All linear terms are in cocycle, with $y$, gives cycles of the form  
\begin{center}
$\left( x_i\wedge z_i+x_j\wedge z_j\right)\wedge z_a\wedge z_b$
\end{center}
 $i,j$ different than $a,b$ and $z_i=[x_1,\ldots,\widehat{x_i},\ldots,x_k,y]$. There are only k many terms. 
\item all but one linear terms, no y are in the kernel.
\item all but one linear terms with y, cycle of degree 2k-1 terms, there is only one element.
\end{enumerate}
Hence in total, $\vert\im\partial_{2k-1}\vert=k+1+k=2k+1$. Therefore kernel is ${2k+1\choose 2k-1}-(2k+1)$.
\end{proof}
By combining these results, we get:
\begin{theorem}
Homology of free 3-step nilpotent k-ary Lie algebra with $\dim V=k$ is:
\begin{gather*}
\vert H^1(\gG)\vert=k,\hspace{1cm}
\vert H^k(\gG)\vert={2k+1\choose k}-(3k+2),\hspace{1cm}
\vert H^{2k-1}(\gG)\vert=(2k+1)(k-1)
\end{gather*}
\end{theorem}

\paragraph{$k=3$}
Let $\gG$ be $3$-step nilpotent free 3-ary Lie algebra on $V$ with $\dim V=3$. The resolution is very similar to the previous case \ref{f2}, only the length of the complex is $4$, instead of $3$. We use the following Hall basis:\\
There are $3$ elements $x_1,x_2,x_3$ of degree 1\\
There is $1$ element $y=[x_1, x_2, x_3]$ of degree $3$\\
There are $3$ elements $[x_1, x_2,y], [x_1, x_3,y], [x_2,x_3,y]$ of degree $5$.

Chevalley-Eilenberg complex of $\gG$ is:
\[\xymatrixcolsep{5pt}
\xymatrix{0\ar[rr]&&\bigwedge^{7}\gG\ar[rrrr]^{\partial_7}&&&&\bigwedge^{5}\gG\ar[rrrr]^{\partial_{5}} &&&& \bigwedge^{3}\gG\ar[rrrr]^{\partial_{3}}  &&&&  \gG \ar[rr] && \Cc
}\]

Image of $\partial_3$ is of size $4$, hence kernel of $\partial_3$ is ${7\choose 3}-4=31$. 
Elements of $\bigwedge^{5}\gG$ which does not contain at least $2$ degree one terms are in the kernel. Hence we have four cases for the image of $\partial_5$:
\begin{enumerate}[label=\roman*)]
\item All linear terms are in cycle, no $y$, gives cycles of $y$ and degree $5$ elements, total size is $3$
\item All linear terms are in cocycle, with $y$, gives elements of the form:
$\partial(x_1\wedge x_2\wedge x_3\wedge y\wedge z_1)=-x_3\wedge z_1\wedge z_3+x_2\wedge z_1\wedge z_2$\\
$\partial(x_1\wedge x_2\wedge x_3\wedge y\wedge z_2)=x_3\wedge z_2\wedge z_3+x_1\wedge z_1\wedge z_2$\\
$\partial(x_1\wedge x_2\wedge x_3\wedge y\wedge z_3)=-x_2\wedge z_2\wedge z_3-x_1\wedge z_1\wedge z_3$\\
They are linearly independent, hence image is $3$ dimensional for these cycles.  
\item all but one linear terms, no y are in the kernel.
\item all but one linear terms with y: cycle of degree $5$ terms, there is only one element.
\end{enumerate}
Hence in total, $\vert\im\partial_{5}\vert=7$. Image of $\partial_7$ is zero which is obvious. By combining these results, we get:

\begin{theorem}\label{thm3}
Homology of free 3-step nilpotent 3-ary Lie algebra with $\dim V=3$ is:
\begin{gather*}
\vert H^1(\gG)\vert=3,\hspace{1cm}
\vert H^3(\gG)\vert=24,\hspace{1cm}
\vert H^{5}(\gG)\vert=14\hspace{1cm}
\vert H^{7}(\gG)\vert=1
\end{gather*}
\end{theorem}

\section{Free 2-step Nilpotent k-ary Lie Algebras}\label{sectionFree}
We divide this section into three parts. First we discuss homology of free 2-step nilpotent k-ary Lie algebras as $GL(V)$ representations, and state a result on the second homology when $k=3$. In the second part, we obtain lower bound of Betti numbers. We discuss representation stability of these algebras at the end of the section.

\subsection{Properties} We recall the definition \ref{deffree2} of free 2-step nilpotent k-ary Lie algebra: $\mathfrak{n}^k_2=V\oplus\bigwedge^kV$, $\dim_{\Cc} V=n\geq k$ with respect to bracket:
\begin{center}
$\mathfrak{n}^k_2\times\cdots\times\mathfrak{n}^k_2\rightarrow \mathfrak{n}^k_2$\\
$[v_1,\ldots,v_k]=\begin{cases} 
v_1\wedge\cdots\wedge v_k &v_i\in V \quad \forall i\\
0 &\text{otherwise}
\end{cases}$
\end{center}
We also recall the complex \ref{defcomplex}:

\begin{align}\label{c2}
\xymatrixcolsep{5pt}
\xymatrix{& ...\bigwedge^{3k-2}\mathfrak{n}^k_2\ar[rrrr]^{\partial_{3k-2}}  &&&& \bigwedge^{2k-1}\mathfrak{n}^k_2\ar[rrrr]^{\partial_{2k-1}} &&&& \bigwedge^{k}\mathfrak{n}^k_2\ar[rrrr]^{\partial_{k}}  &&&&  \mathfrak{n}^k_2\ar[rr] &&\Cc
}
\end{align}
with the differential \ref{defdiff}:
\begin{align*}
\partial_t(x_1\wedge\ldots\wedge x_t)=\sum\limits_{\sigma\in D_{t,k}}sgn(\sigma)[x_{\sigma(1)},\ldots,x_{\sigma(k)}]\wedge x_{\sigma(k+1)}\wedge \cdots \wedge x_{\sigma(t)}
\end{align*}

\begin{lemma} Homology of the complex \ref{c2} is $GL(V)$ module.
\end{lemma}

\begin{proof}
Since $\mathfrak{n}^k_2=V\oplus\bigwedge^k V$ is representation of $GL(V)$, kernels and the images of boundary map $\partial$ are also representations of $GL(V)$, which makes each homology $GL(V)$ representation. In details, the action of $g\in GL(V)$ on $\mathfrak{n}^k_2$ is
\begin{align}
g[x_1,\ldots,x_k]=g\left(\sum_{\sigma\in\mathfrak{S}_k}\sign(\sigma)x_{\sigma(1)}\otimes\cdots\otimes x_{\sigma(k)}\right)=\sum_{\sigma}\sign(\sigma)gx_{\sigma(1)}\otimes\cdots\otimes gx_{\sigma(k)}=[gx_1,\ldots,gx_k]
\end{align}

and on cycles $\bigwedge^a\mathfrak{n}^k_2$ is $g.(v_1\wedge\ldots\wedge v_a)=gv_1\wedge\ldots\wedge gv_a$
\end{proof}

\begin{remark} There are many works for the case $k=2$ \cite{sigg1996laplacian}, \cite{jozefiak1988representation}. First of all, the plethsym $\bigwedge^j\left(\bigwedge^2V\right)$ is well-known \cite{landsberg}. Secondly, multiplicity of every irreducible summand in $\bigwedge\left(V\oplus\bigwedge^2V\right)$ is one. Combining these together with the result of Kostant \cite{kostant1963lie}, homology of free 2-step nilpotent Lie algebra is given by
\begin{align}\label{thmsigg}
H^p(\mathfrak{n}^2_2)=\bigoplus\limits_{\substack{\lambda=\lambda^{t}\\ \lambda\vdash p}} S_{\lambda}V
\end{align}
where $\lambda$ is self-conjugate partition of size $p$ \cite{sigg1996laplacian}, \cite{sam2013homology}.
\end{remark}

When $k\geq 3$, there is not any result relating homology of nilpotent k-ary algebras to eigenvalues of Lie algebra Laplacian  as in \cite{kostant1963lie}. Moreover, there is no formula to compute the plethsym $\bigwedge^j\left(\bigwedge^kV\right)$ $j,k\geq 3$. Therefore an analogous result to \ref{thmsigg} is a difficult one. Nevertheless we prove:

\begin{proposition} $H^{(3)}\left(\mathfrak{n}^3_2\right)=S_{221}V\oplus S_{2111}V\oplus S_{3211}V$
\end{proposition}
\begin{proof}
We will compute the kernel of $\partial_3$ and image of $\partial_5$ in the sequence \ref{defcomplex}. It is clear that every cycle of $\bigwedge^3\mathfrak{n}^3_2$ except $\bigwedge^3V$ is in the kernel of $\partial_3$. Therefore :

\begin{align}\label{k1}
\ker\partial_3=\bigwedge^2V\otimes\bigwedge^3V\bigoplus V\otimes\bigwedge^2\left(\bigwedge^3V\right)\bigoplus\bigwedge^3\left(\bigwedge^3V\right)
\end{align}
Before computing the image of $\partial_5$, we list some formulas which are consequences of Littlewood-Richardson and plethysm rules \cite{landsberg}:

\begin{gather}\label{k2}
\bigwedge^2V\otimes\bigwedge^3V\cong \bigwedge^5V\oplus S_{2111}V\oplus S_{221}V \\
\label{k3} V\otimes\bigwedge^2\left(\bigwedge^3V\right)\cong S_{3211}V\oplus S_{2221}V\oplus S_{22111}V\oplus S_{211111}V\oplus \bigwedge^7V \\
\label{k4} \bigwedge^4V\otimes\bigwedge^3V\cong S_{2221}V\oplus S_{22111}V\oplus S_{211111}V\oplus\bigwedge^7V
\end{gather}

Image of $\partial_5$ is
\begin{gather}
\im\partial_5=\bigwedge^5V\oplus\left(\bigwedge^4V\otimes\bigwedge^3V\bigcap V\otimes\bigwedge^2\left(\bigwedge^3V\right)\right)\oplus \bigwedge^3\left(\bigwedge^3V\right)\\
=S_{2221}V\oplus S_{22111}V\oplus S_{211111}V\oplus\bigwedge^7V \oplus \bigwedge^3\left(\bigwedge^3V\right)\label{k5}
\end{gather}
By formulas \ref{k1},\ref{k2},\ref{k3},\ref{k4} and \ref{k5}, homology is 
\begin{gather}
H^{(3)}(\mathfrak{n}^3_2)=\cfrac{\ker\partial_3}{\im\partial_5}=S_{221}V\oplus S_{2111}V\oplus S_{3211}V
\end{gather}

\end{proof}

\subsection{Lower Bounds of Betti Numbers}
In this part, we study homology and lower bounds of Betti rank of free 2-step nilpotent k-ary Lie algebras. We start with an easy observation:

\begin{lemma} The decomposition of $\bigwedge^{ik-(i-1)}\mathfrak{n}^k_2$ is
\begin{gather}
\bigwedge^{ik-(i-1)}\mathfrak{n}^k_2=\bigoplus\limits^{ik-(i-1)}_{j=0}\left(\bigwedge^{ik-(i-1)-j}V\otimes\bigwedge^j\left(\bigwedge^kV\right)\right)
\end{gather}
\end{lemma}

\begin{proof}
It follows from Cauchy formula.
\end{proof}

\begin{lemma} Boundary maps are degree preserving, and their restrictions on each degree is:
\begin{gather}
\partial^j_{ik-(i-1)}:\bigwedge^{ik-(i-1)-j}V\otimes\bigwedge^j\left(\bigwedge^kV\right)\longrightarrow \bigwedge^{(i-1)k-i-j+1}V\otimes\bigwedge^{j+1}\left(\bigwedge^{k}V\right)
\end{gather}
\end{lemma}
\begin{proof}
Since center of $\mathfrak{n}^k_2$ is $\bigwedge^k V$, only nonzero bracket is of the form $[x_1\ldots,x_k]$, $x_i\in V$. Boundary map $\partial$ takes $k$ pairs from $\bigwedge^{ik-(i-1)-j}$, and adds $1$ element to $\bigwedge^j\left(\bigwedge^k V\right)$.
\end{proof}

\begin{proposition}\label{homologydecomposition} Homology of complex at $ik-(i-1)$ is 
\begin{gather}
H^{ik-(i-1)}\left(\gG\right)=\cfrac{\ker\partial_{ik-(i-1)}}{im\partial_{(i+1)k-i}}=\bigoplus\limits^{ik-(i-1)}_{j=0}\cfrac{\ker\partial^j_{ik-(i-1)}}{\im\partial^j_{(i+1)k-i}}
\end{gather}
\end{proposition}
\begin{proof}
Cycles of the form $\bigwedge^mV\otimes\bigwedge^{(i+1)k-i-m}\left(\bigwedge^kV\right)$ maps to zero when $0\leq m<k$, so they are not in the image. The nonzero indices are same both boundary maps $\partial_{ik-(i-1)}$ and $\partial_{(i+1)k-i}$ in the homology.
\end{proof}

\begin{corollary} Betti rank is :
\begin{gather}
\left|H^{ik-(i-1)}\left(\mathfrak{n}^k_2\right)\right|=\sum\limits^{ik-(i-1)}_{j=0}\left|\ker\partial^j_{ik-(i-1)}\right|-\left|\im\partial^j_{(i+1)k-i}\right|
\end{gather}
\end{corollary}

\begin{proposition}
\begin{gather}
\left|H^{ik-(i-1)}\left(\mathfrak{n}^k_2\right)\right|=\sum\limits^{ik-(i-1)}_{j=0}\left|\ker\partial^j_{ik-(i-1)}\right|-\left|\im\partial^j_{(i+1)k-i}\right|\\
\geq \sum\limits^{(i-1)(k-1)}_{j=1}\left|\ker\partial^j_{ik-(i-1)}\right|-\left|\im\partial^j_{(i+1)k-i}\right|
\end{gather}
\end{proposition}
\begin{proof}
Cycles of the form 
\begin{align}
\bigwedge^{ik-(i-1)-j}V\otimes\bigwedge^{j}\left(\bigwedge^{k}V\right) 
\end{align} with $(i-1)(k-1)<j$ are in the kernel of $\partial^j_{ik-(i-1)}$. Hence, if we ignore their contribution to homology, we get smaller rank. \\
Also, $\bigwedge^{ik-(i-1)}V$ is always a submodule of $\bigwedge^{(i-1)k-(i-1)}V\otimes\bigwedge^kV$ by Pieri's rule. This gives the claimed lower bound. 
\end{proof}

\begin{lemma}\label{pieri1}$\bigwedge^kV\otimes\bigwedge^{\alpha}\left(\bigwedge^kV\right)\cong \bigwedge^{\alpha+1}\left(\bigwedge^kV\right)\bigoplus S_{21^{\alpha-1}}\left(\bigwedge^kV\right)$
\end{lemma}
\begin{proof}
By Pieri's rule, $W\otimes\bigwedge^2W\cong \bigwedge^3W\oplus S_{21}W$. Replacing $W$ by $\bigwedge^kV$ proves the lemma.
\end{proof}

\begin{theorem}\label{thmLowerbound}
The rank of the homology at $ik-(i-1)$, $i\geq 2$, $H^{ik-(i-1)}\left(\mathfrak{n}^k_2\right)$ is bounded below by
\begin{gather}
\left|H^{ik-(i-1)}\left(\mathfrak{n}^k_2\right)\right|\geq{n\choose k}{x\choose (i-1)(k-1)}-{n\choose 2k}{x\choose (i-1)(k-1)-1}-{n\choose 0}{x\choose (i-1)(k-1)+1}\nonumber
\end{gather}
where $n$ is the dimension of $V$.
\end{theorem}
\begin{proof}
By proposition \ref{homologydecomposition}, homology of the sequence :

\begin{gather}
\label{h1}\bigwedge^{2k}V\otimes\bigwedge^{(i-1)(k-1)-1}\left(\bigwedge^kV\right)\longrightarrow \bigwedge^{k}V\otimes\bigwedge^{(i-1)(k-1)}\left(\bigwedge^k V\right)\longrightarrow\bigwedge^{(i-1)(k-1)+1}\left(\bigwedge^k V\right)
\end{gather}
contributes to the homology of the complex \ref{c2}.
By the lemma \ref{pieri1}, the middle term decomposes as:
\begin{gather}
\bigwedge^{2k}V\otimes\bigwedge^{(i-1)(k-1)-1}\left(\bigwedge^kV\right)\longrightarrow S_{\lambda}\left(\bigwedge^kV\right)\bigoplus \bigwedge^{(i-1)(k-1)+1}\left(\bigwedge^k V\right)\longrightarrow\bigwedge^{(i-1)(k-1)+1}\left(\bigwedge^k V\right)
\end{gather}
where $\lambda=\left(21^{(i-1)(k-1)-1}\right)$. Therefore homology $H^{ik-(i-1)}\mathfrak{n}^k_2$ contains 
\begin{gather}
\coker\left(\bigwedge^{2k}V\otimes\bigwedge^{(i-1)k-(i)}V\longrightarrow S_{21^{(i-1)k-i}}V\right)
\end{gather}
as submodule. In particular, homology of the sequence \ref{h1} is bounded by 
\begin{gather}
{n\choose k}{x\choose (i-1)(k-1)}-{n\choose 2k}{x\choose (i-1)(k-1)-1}-{n\choose 0}{x\choose (i-1)(k-1)+1}\nonumber
\end{gather}
Claim follows.
\end{proof}

\begin{proposition} The lower bound in theorem \ref{thmLowerbound} is asymptotically
\begin{align}
\cfrac{e^{2k}n^{2k}}{2k^{2k}}\left(\cfrac{1}{\pi k}-\cfrac{(\alpha+1)}{\sqrt{\pi k}}\cfrac{1}{2^{2k}}\right){x\choose \alpha}\cfrac{1}{(\alpha+1)(x-\alpha+1)}
\end{align}
with $\alpha=(i-1)(k-1)$
\end{proposition}
\begin{proof}
Assume that $n$ is large than $k$. Then we can use the approximation:
\begin{gather}
{n\choose k}\approx \cfrac{e^k}{\sqrt{2\pi k}}\left(\cfrac{n}{k}-\cfrac{1}{2}\right)^k
\end{gather}
We start with the bound in the Theorem \ref{thmLowerbound}, and simplify:
\begin{gather}
{x\choose \alpha}{n\choose k}-{x\choose \alpha-1}{n\choose 2k}-{x\choose \alpha+1}{n\choose 0}=\\
{x\choose \alpha}{n\choose k}-\cfrac{\alpha}{x-\alpha+1}{x\choose \alpha}{n\choose 2k}-\cfrac{x-\alpha}{\alpha+1}{x\choose \alpha}{n\choose 0}
\end{gather}
We approximate the coefficient:
\begin{gather}
\cfrac{\alpha}{\alpha+1}x+\cfrac{\alpha}{\alpha+1}-{n\choose 2k}\cfrac{\alpha}{x-\alpha+1}\implies\\
x^2+(2-\alpha)x+1-\alpha-(\alpha+1){n\choose 2k}\approx\\
\cfrac{e^{2k}}{2\pi k}\left(\cfrac{n}{k}-\cfrac{1}{2}\right)^{2k}+(2-\alpha)\cfrac{e^k}{\sqrt{2\pi k}}\left(\cfrac{n}{k}-\cfrac{1}{2}\right)^k+1-\alpha-(\alpha+1)\cfrac{e^{2k}}{\sqrt{4\pi k}}\left(\cfrac{n}{2k}-\cfrac{1}{2}\right)^{2k}
\end{gather}
\text{The leading term is}
\begin{gather}
\cfrac{e^{2k}n^{2k}}{2k^{2k}}\left(\cfrac{1}{\pi k}-\cfrac{(\alpha+1)}{\sqrt{\pi k}}\cfrac{1}{2^{2k}}\right)
\end{gather}
Therefore lower bound asymptotically is

\begin{align}
\cfrac{e^{2k}n^{2k}}{2k^{2k}}\left(\cfrac{1}{\pi k}-\cfrac{(\alpha+1)}{\sqrt{\pi k}}\cfrac{1}{2^{2k}}\right){x\choose \alpha}\cfrac{1}{(\alpha+1)(x-\alpha+1)}
\end{align}
In particular, positivity of the leading term is because of $2^{2k}\geq ki\sqrt{\pi k}$. This finishes the proof.
\end{proof}

For the second homology, we can track another approach.
\begin{proposition}\label{prop2}
 The second homology of free 2-step nilpotent k-ary Lie algebra always contains 
\begin{gather}
 \bigoplus\limits^{k-1}_{j=1}S_{2^j1^{2k-2j+1}}V \subset H^{k}\left(\mathfrak{n}^k_2\right)
\end{gather}
\end{proposition}

\begin{proof}
By proposition \ref{homologydecomposition}, the second homology i.e. $i=1$ contains $\cfrac{\ker\partial^0_{ik-(i-1)}}{\im\partial^0_{(i+1)k-i}}$, which is :
\begin{gather}
 \cfrac{\bigwedge^{k-1}V\otimes\bigwedge^{k}V}{\bigwedge^{2k-1}V}= \bigoplus\limits^{k-1}_{j=1}S_{2^j1^{2k-2j-1}}V 
\end{gather}
\end{proof}

\begin{corollary} The second homology of free 2-step nilpotent $k$-ary Lie algebra satisfies:
\begin{align}
\vert H^{(k)}\left(\mathfrak{n}^k_2\right)\vert\geq {n\choose k}{n\choose k-1}-{n\choose 2k-1}
\end{align}
\end{corollary}

\subsection{Representation stability}\label{sectionRep}

Representation stability is currently an active field and find wide range of applications in many branches including Lie algebra homology. For example, homologies of free nilpotent Lie algebras are representation stable \cite{church2013representation}, which follows representation stability of free Lie algebras. Multiplicity of irreducible representations in \ref{freelie} is given by the formula:
\begin{center}
$c_{\lambda}=\cfrac{1}{n}\sum\limits_{d\vert n}\mu(d)\chi_{\lambda}(\tau^{n/d})$
\end{center}
where $\chi_{\lambda}$ is the character of $\mathfrak{S}_n$-representation $\lambda$ and $\tau$ is the cycle $(12\ldots n)$ \cite{klyachko1974lie}.
For free k-ary Lie algebras, we do not have such a formula yet, since which representations occur in it is still an open problem \cite{friedmann2017generalization}.

\begin{theorem}
 Homologies of free nilpotent k-ary Lie algebras are representation stable.
\end{theorem}

\begin{proof}
 
Instead of representation stability techniques\footnotemark, we sketch direct arguments of \cite{tirao2002homology}\footnotetext{Representation stability of free k-ary Lie algebras will be a separate work}. Let $\Cc^n\hookrightarrow\Cc^{n+1}$ be an embedding. This induces the following inclusions:
\begin{enumerate}[label=\roman*)]
\item Free k-ary Lie algebras: $F^{(k)}(\Cc^n)\hookrightarrow F^{(k)}(\Cc^{n+1})$
\item Free s-step nilpotent k-ary Lie algebras $\mathfrak{n}^k_s(\Cc^n)\hookrightarrow\mathfrak{n}^k_s(\Cc^{n+1})$
\item Terms in the complex \ref{defcomplex}: $\bigwedge^{\alpha}\mathfrak{n}^k_s(\Cc^n)\hookrightarrow\bigwedge^{\alpha}\mathfrak{n}^k_s(\Cc^{n+1})$
\item In particular, $GL_n\mapsto GL_{n+1}$ via $A\rightarrow \begin{bmatrix}
A&0\\
0&1
\end{bmatrix}$
\end{enumerate}
We denote $\mathfrak{n}^k_s(\Cc^n)$ and $\mathfrak{n}^k_s(\Cc^{n+1})$ by $\mathfrak{n}_1$ and $\mathfrak{n}_2$ respectively. We compare homology $H^{\alpha}(n_1)$ and $H^{\alpha}(n_2)$.
If nonzero $[\lambda]\in H^{\alpha}(n_1)$, then $[\lambda]\in H^{\alpha}(n_2)$, since $[\lambda]$ is in the kernel of $\partial^2_{\alpha}$ via Young diagram $\left(\lambda,0\right)$. If $[\lambda]$ is in the image of $\partial^2_{\alpha+k-1}$ i.e. $[\lambda]=\partial^2_{\alpha+k-1}(\omega)$, then $\omega=[\left(\lambda,0\right)]$ by using $GL$ morphisms and in particular $\omega\in\bigwedge^{\alpha+k-1}n_2$. The restriction of $\omega$ gives either zero or $[\lambda]$. Therefore for large enough dimensions, $H^{\alpha}(n_1)$ is equal to $H^{\alpha}(n_2)$, which implies representation stability.
\end{proof}

\begin{example}
The second homology of $\mathfrak{n}^3_2$ is representation stable:
\begin{itemize}
\item $\dim V=3$, $H^{(3)}\left(\mathfrak{n}^3_2\right)= {\tiny\Yvcentermath1 \yng(2,2,1)}$ its dimension is $3$.
\item $\dim V\geq 4$ , $H^{(3)}\left(\mathfrak{n}^3_2\right)= {\tiny\Yvcentermath1 \yng(2,2,1)\oplus\yng(2,1,1,1)\oplus \yng(3,2,1,1)}$
\end{itemize}
\end{example}

It would be nice to get another stability in the direction of $k$, i.e. ith homologies of nilpotent k-ary Lie algebras while $k$ is increasing. However, except the first homology i.e. $H^1(\gG)=V$, this cannot not true, simply by the result of proposition \ref{prop2}:
\begin{gather}
 \cfrac{\bigwedge^{k-1}V\otimes\bigwedge^{k}V}{\bigwedge^{2k-1}V}= \bigoplus\limits^{k-1}_{j=1}S_{2^j1^{2k-2j-1}}V 
\end{gather}
is always a submodule of the second homology of $\mathfrak{n}^k_2$, and not stabilizing.

\section{Toral Rank conjecture}\label{sectionToral}
Let $\mathfrak{n}$ be nilpotent Lie algebra with center $\mathfrak{z}$. Toral rank conjecture states that:
\begin{align}
\sum\limits^{\dim\mathfrak{n}}_{i=0} H^i\left(\mathfrak{n}\right)\geq 2^{\dim\mathfrak{z}}
\end{align}
We carry it into k-ary setup:
\begin{theorem} Toral rank conjecture  holds for  free 2-step nilpotent k-ary Lie algebras, k-ary Heisenberg Lie algebras, k-ary ACJ Lie algebras and free 3-step nilpotent k-ary Lie algebras   of small dimension.
\end{theorem}
\begin{proof}
In theorem \ref{thmRefinement}, we give a refinement of toral rank conjecture for 2-step nilpotent k-ary Lie algebras. Hence the only remaining cases are small dimensional 3-step nilpotent k-ary Lie algebras studied in \ref{homfree3}. The sum of the homologies when $k\geq 4$ is:
\begin{gather}
k+{2k+1\choose k}-(k+1)-(2k+1)+(2k+1)(k-1)=\\
{2k+1\choose k}+2k^2-3k-3
\end{gather}

It is clear that
\begin{align}
{2k+1\choose k}+2k^2-3k-3>{2k\choose k}\geq \cfrac{(2k)^k}{k^k}=2^k
\end{align}
Since the size of the center is $k$, toral rank conjecture holds for $k\geq 4$. If $k=3$, the total homology is $42$ according to the Theorem \ref{thm3} and greater than $8$.
\end{proof}

\subsection{Refinement of Toral rank Conjecture}

We prove the following:
\begin{theorem}\label{thmRefinement}
The total homology of any 2-step nilpotent $k$-ary Lie algebra $\gG$ satisfies

\begin{align}
\vert H^*(\gG)\vert\geq \sum\limits^{k-1}_{i=0}\left| \sum_{j=0}\left(-1\right)^{j}{\vert\mathfrak{v}\vert\choose kj+i}\right|2^{\vert\mathfrak{z}\vert}
\end{align}
where $\mathfrak{z}$ is the center of $\gG$ and $\mathfrak{v}$ is its vector space complement. 
\end{theorem}

\begin{proof}
Let $\gG=\mathfrak{z}\oplus\mathfrak{v}$ as vector space, the complex \ref{defcomplex} becomes $\bigwedge\gG=\bigwedge\left(\mathfrak{z}\otimes\mathfrak{v}\right)$. Boundary maps are $\partial:\bigwedge^p\mathfrak{v}\otimes\bigwedge^q\mathfrak{z}\rightarrow
\bigwedge^{p-k}\mathfrak{v}\otimes\bigwedge^{q+1}\mathfrak{z}$. This implies that the complex $(\bigwedge\gG,\partial)$ is direct sum of complexes $(\bigwedge^{kj+i}\mathfrak{v}\otimes\bigwedge \mathfrak{z},\partial)$. 

Homology of any finite complex satisfies:
\begin{gather}
\vert H^*(\gG)\vert\geq \left| \sum_{a=0}\left(-1\right)^{a}\dim\left(\bigwedge^{a}\gG\right)\right|
\end{gather}
Now, we replace $\gG$ by $\mathfrak{v}\oplus\mathfrak{z}$ and simplify the expression since the total homology is direct sum of all homologies in  mod k:
\begin{gather}
\vert H^*(\gG)\vert\geq \sum\limits^{k-1}_{i=0}\left| \sum_{j=0}\left(-1\right)^{j}\dim\left(\bigwedge^{kj+i}\mathfrak{v}\otimes\bigwedge\mathfrak{z}\right)\right|\\
\vert H^*(\gG)\vert\geq \sum\limits^{k-1}_{i=0}\left| \sum_{j=0}\left(-1\right)^{j}\dim\left(\bigwedge^{kj+i}\mathfrak{v}\right)\right|2^{\vert\mathfrak{z}\vert}\\
\vert H^*(\gG)\vert\geq \sum\limits^{k-1}_{i=0}\left| \sum_{j=0}\left(-1\right)^{j}{\vert\mathfrak{v}\vert\choose kj+i}\right|2^{\vert\mathfrak{z}\vert}
\end{gather}

\end{proof}

\begin{corollary} The refinement of the toral rank conjecture for 2-step nilpotent k-ary Lie algebras is
\begin{align*}
\vert H^*(\gG)\vert\geq \sum\limits^{k-1}_{i=0}\left| \sum_{j=0}\left(-1\right)^{j}{\vert\mathfrak{v}\vert\choose kj+i}\right|2^{\vert\mathfrak{z}\vert}>2^{\vert\mathfrak{z}\vert}
\end{align*}
When $k=2$, one get explicit formula obtained in \cite{tirao2000refinement}.

\end{corollary}

\begin{remark}When we apply this to 5-ary Heisenberg Lie algebra with $m=1$, complex is 
\begin{gather}
0\longrightarrow \bigwedge^4\hH\longrightarrow\hH\longrightarrow 0
\end{gather}
and the total homology: $1+5+5=11$. The current 5-ary Lie algebra $\hH\otimes\cfrac{\Cc[t]}{t^2}$ is 2-step nilpotent, and its dimension is $12$. The lower bound of the total homology of any $12$ dimensional 2-step nilpotent 5-ary Lie algebra is $2900$ by the formula we obtained in the Theorem \ref{thmRefinement}.
Property M \ref{defPropertyM} asserts that
\begin{align}
\left|H\left(\hH\otimes \faktor{\Cc[t]}{t^j}\right)\right|=\left|H\left(\hH\right)\right|^{j}
\end{align}
However, at $j=2$, claimed total homology is $11^2=121<2900$
 which is impossible.\\
This suggests that for k-ary Lie algebras, exponent of the total homology might be $(k-1)j$ in \ref{defPropertyM}.
\end{remark}
Here we add comparison of the lower bounds according to Theorem \ref{thmRefinement} and exponents of $2$. The entries of the third column is due to the formula in \cite{tirao2000refinement}.

$\left[ \begin {array}{ccccccccc} \hline
n&k=2&\log_2&k=3&\log_2&k=4&\log_2&k=5&\log_2\\\hline
1&2& 1.0&2& 1.0&2& 1.0&2& 1.0\\\noalign{\medskip}2&2& 1.0&4& 2.0&4& 2.0&4& 2.0\\\noalign{\medskip}3
&4& 2.0&6& 2.584962500&8& 3.0&8& 3.0\\\noalign{\medskip}4&4& 2.0&12&
 3.584962501&14& 3.807354922&16& 4.0\\\noalign{\medskip}5&8& 3.0&18&
 4.169925001&28& 4.807354922&30& 4.906890596\\\noalign{\medskip}6&8&
 3.0&36& 5.169925000&48& 5.584962501&60& 5.906890595
\\\noalign{\medskip}7&16& 4.0&54& 5.754887502&96& 6.584962500&110&
 6.781359713\\\noalign{\medskip}8&16& 4.0&108& 6.754887502&164&
 7.357552004&220& 7.781359713\\\noalign{\medskip}9&32& 5.0&162&
 7.339850002&328& 8.357552004&400& 8.643856190\\\noalign{\medskip}10&
32& 5.0&324& 8.339850002&560& 9.129283017&800& 9.643856190
\\\noalign{\medskip}11&64& 6.0&486& 8.924812503&1120& 10.12928302&1450
& 10.50183718\\\noalign{\medskip}12&64& 6.0&972& 9.924812502&1912&
 10.90086681&2900& 11.50183718\\\noalign{\medskip}13&128& 7.0&1458&
 10.50977500&3824& 11.90086681&5250& 12.35810171\\\noalign{\medskip}14
&128& 7.0&2916& 11.50977500&6528& 12.67242534&10500& 13.35810171
\\\noalign{\medskip}15&256& 8.0&4374& 12.09473750&13056& 13.67242534&
19000& 14.21371180\\\noalign{\medskip}16&256& 8.0&8748& 13.09473750&
22288& 14.44397955&38000& 15.21371180\\\noalign{\medskip}17&512& 9.0&
13122& 13.67970001&44576& 15.44397955&68750& 16.06907210
\\\noalign{\medskip}18&512& 9.0&26244& 14.67970000&76096& 16.21553300&
137500& 17.06907210\\\noalign{\medskip}19&1024& 10.0&39366&
 15.26466251&152192& 17.21553300&248750& 17.92433701
\\\noalign{\medskip}20&1024& 10.0&78732& 16.26466251&259808&
 17.98708633&497500& 18.92433701\end {array} \right]$

\bibliographystyle{alpha}

\end{document}